\documentclass{article}

\usepackage{geometry}
\usepackage[all]{xy}
\usepackage{amsmath}
\usepackage{amssymb}
\usepackage{amsthm}
\usepackage{hyperref}
\title{\normalsize \bf REPRESENTING FINE SHAPE OF LOCAL COMPACTA BY HOMOTOPY CLASSES OF ORDINARY MAPS}
\author{\footnotesize VLADISLAV ZEMLYANOY \\
  \footnotesize {\it Doctoral School in Mathematics, Higher School of Economics (HSE University)} \\
  \footnotesize {\it Moscow, Russia} \\
  \footnotesize {\it tagambit@yandex.ru} \\
  \footnotesize {\it ORCID} 0000-0001-7450-5325
}
\date{}

\begin{document}
\def \R {\mathbb{R}}
\def \N {\mathbb{N}}
\def \i {\imath}
\def \j {\jmath}

\setlength{\parskip}{0.25em}

\newcounter{count}[section]
\renewcommand{\thecount}{\arabic{section}.\arabic{count}}

\theoremstyle{plain}
\newtheorem{theorem}[count]{Theorem}
\newtheorem*{theorem*}{Theorem}
\newtheorem{lemma}[count]{Lemma}
\newtheorem{proposition}[count]{Proposition}
\newtheorem{corollary}[count]{Corollary}

\theoremstyle{definition}
\newtheorem{definition}[count]{Definition}
\newtheorem*{convention}{Convention}
\newtheorem*{notation}{Notation}
\newtheorem{example}[count]{Example}

\theoremstyle{remark}
\newtheorem{remark}[count]{Remark}

\maketitle

\begin{abstract}
\noindent 
Fine shape, as defined by Melikhov, is an extension of the strong shape 
category of compacta (compact metrizable topological spaces) to all 
metrizable spaces, notable for being compatible with both \v{C}ech 
cohomology and Steenrod-Sitnikov homology. In this work we study fine 
shape of local compacta (locally compact separable metrizable spaces), 
and construct, for every local compactum $X$, a space $|X|$ unique up 
to a homotopy equivalence and such that fine shape classes from any 
locally compact metrizable space $Y$ to $X$ bijectively correspond to 
homotopy classes of ordinary maps from $Y$ to $|X|$. This correspondence 
is (contravariatly) functorial in $Y$, thus giving a representation of 
$Y$-dependent contravariant functor for a fixed $X$; the universal class 
corresponding to the identity map of $X$ is the homotopy class of a 
specific embedding of $X$ into $|X|$ that is a fine shape equivalence.
\end{abstract}

\section{Introduction}
\label{sec_intro}

The strong shape category of compact metrizable topological spaces 
(compacta) is well-known, defined in multiple ways that all turn out to 
be 
equivalent~\cite{Bauer1978, Kodama1979, Calder1981, Cathey1981, DydakSegal1981}; 
its theory and applications are extensive, and still being researched 
(for example,~\cite{Mrozik2008,Baladze2020}). In the case of noncompact 
spaces, strong shape is less extensively studied or applied. Strong 
shape category even of all topological spaces can be 
defined~\cite{Batanin1997,Mardesic2000} and has interesting properties 
of its own, though the complexity inherent in the definition and 
computations seems to have limited its usage in practice.

Fine shape, as defined by Melikhov~\cite{Melikhov2022F}, is a different 
extension of strong shape from compacta to all metrizable topological 
spaces. Notably, as~\cite{Melikhov2022F} proves, a (co)homology theory 
is fine shape invariant if and only if it satisfies the map excision 
axiom (that is, for any closed map $f\colon (X,A) \to (Y,B)$ between 
pairs of spaces, if $f$ restricts to an homeomorphism of $X\setminus A$ 
onto $Y\setminus B$, then $f$ induces an isomorphism 
$H_{n}(X,A) \simeq H_{n}(Y,B)$ for all $n$); 
this includes both \v{C}ech cohomology and Steenrod-Sitnikov homology 
(defined as the direct limit of Steenrod homology of compacta; see 
footnote~3 in~\cite{Melikhov2022F}). These two theories, in turn, seem 
to be a natural choice of homology and cohomology for metrizable spaces, 
and especially for local compacta --- that is, metrizable topological 
spaces that are locally compact and separable (equivalently, 
second-countable); see~\cite{Sklyarenko1979}. 
This makes fine shape a strong candidate for a good homotopy theory of 
metrizable spaces; what makes it potentially useful, however, is the 
simplicity of its definition and application compared to the established 
noncompact strong shape.

The definition of fine shape is based on the following. Every metrizable 
space $X$ can be embedded as a closed subset in a metrizable space $M$ 
that is an absolute retract (equivalently, absolute extensor) with 
respect to metrizable spaces: for every closed subset $A \subseteq Z$ 
(with $Z$, and then also $A$, metrizable), every continuous map 
$f\colon A \to M$ can be extended to a continuous map 
$\bar{f}\colon Z \to M$. Moreover, it can be done so that $X$ is 
{\it homotopy negligible} in $M$, meaning that there exists a homotopy 
$H\colon M \times [0,1] \to M$ with 
$H_{0} = id_{M}$ and $H(M \times (0,1]) \subseteq M\setminus X$. Now we 
have the following

\begin{definition}
Let two metrizable spaces $X$ and $Y$ be embedded as closed homotopy 
negligible subsets in absolute retracts $M$ and $N$ respectively. A 
continuous map $\phi\colon M\setminus X \to N\setminus Y$ is called 
{\it $X-Y$-approaching} if for every sequence $\{p_{n}\}$ of points of 
$M\setminus X$ that has an accumulation point in $X$, the sequence 
$\{\phi(p_{n})\}$ (of points of $N\setminus Y$) has an accumulation 
point in $Y$.
\end{definition}

\noindent This definition clearly looks like an extension of the 
usual notion of continuous map, and approaching maps can be obtained 
from ordinary maps from $X$ to $Y$:

\begin{proposition}(Lemma~\ref{extend_map_to_fSh})
Let $f\colon X \to Y$ be a continuous map between metrizable spaces $X$ 
and $Y$, and let the those be embedded as closed homotopy negligible 
subsets in absolute retracts $M$ and $N$ respectively. Then 

(1)There eixsts an $X-Y$-approaching map 
$\phi\colon M\setminus X \to N\setminus Y$ such that $\phi$ and $f$ 
combine in a continuous map from $M$ to $N$;

(2)For any two homotopic maps $f,f'\colon X \to Y$, any two corresponding maps 
$\phi,\phi'$ given by (1) are homotopic through maps of the same class.
\end{proposition}

\noindent This means that every map (homotopy class of maps, actually) 
from $X$ to $Y$ specifies a homotopy class of $X-Y$-approaching maps 
for any choice of $M$ and $N$. The set of those latter homotopy classes, 
too, turns out to be independent of that choice, depending instead only 
on $X$ and $Y$; this is why we can call every such homotopy class a 
{\it fine shape class} from $X$ to $Y$. Further, we can define 
composition of fine shape classes, identity fine shape classes, and so 
on. From this rough description, it can be seen why fine shape is 
relatively straightforward to work with; the following sections will 
necessarily give examples of actual usage of the definition.

Our goal for the present work is to represent fine shape morphisms via 
(homotopy classes of) ordinary maps. For a restricted class of spaces, 
it turns out to be possible, if notationally cumbersome, to give a 
representation of fine shape morphisms into a given space that is 
functiorial in the domain. To elaborate, we eventually restrict our 
attention to an arbitrary local compactum (that is, a locally compact 
separable metrizable topological space) $X$, and construct a metrizable 
space $|X|$ (which we show to be unique up to a homotopy equivalence) 
such that for all locally compact metrizable spaces $Y$, there is a 
bijection between the set of homotopy classes of continuous maps 
$[Y,|X|]$ and the set of fine shape classes $[Y,X]_{fSh}$, with the 
bijection contravariantly functorial in $Y$.

If we assume $X$ to be a compactum, the construction can be simplified 
(although it is still homotopy equivalent to the general one); moreover, 
$Y$ can then be assumed to be any metrizable space at all. As strong 
shape of compacta is well-studied, there is effectively such a 
construction already given by Cathey~\cite{Cathey1981}, except that it 
is introduced and used there in a different way. On the other hand, our 
general construction of $|X|$ is better demonstrated by first defining 
the simpler compact case, and then modifying it for the broader class of 
spaces.

This dictates the structure of the present work: after stating 
the previously known definitions and theorems we rely on in 
Section~\ref{sec_prelim}, we devote Section~\ref{sec_compact_case} to 
the compact case; there we also explain how our construction of $|X|$ 
relates to the statements of~\cite{Cathey1981}. 
Section~\ref{sec_exh_fun} diverts from the fine shape discussion to 
establish a particular class of functions we use in the general case; 
nothing in it is mathematically challenging, but we feel that explaining 
the notion fully makes the actual general construction clearer, given 
the somewhat convoluted notation involved. Finally, in 
Section~\ref{sec_classifier} we define the space $|X|$ for any local 
compactum $X$, and prove the combination of properties we claim it to 
have. The end result of the present work is the following

\begin{theorem}(Theorem~\ref{th_univ}, Corollary~\ref{cor_class}, Remark~\ref{rem_repr})
For any local compactum $X$, there exists a metrizable space $|X|$ 
unique up to a homotopy equivalence and such that for every locally 
compact metrizable space $Y$, there is a bijection 
$[Y,X]_{fSh} \simeq [Y,|X|]$, contravariantly functorial in $Y$.
\end{theorem}

Other results of further interest include the structure of the space 
$|X|$ itself, showing how to apply this representation in practice, as 
well as the expected universal homotopy class in $[X,|X|]$, which 
corresponds under the bijection to the identity class of $[X,X]_{fSh}$. 
This class is best represented by a particular closed embedding of $X$ 
into $|X|$ that is also a fine shape equivalence; 
from our work in~\cite{Zemlyanoy1}, we call such maps FDR-embeddings 
(see Proposition~\ref{prop_fdr_char}).

Finally, it should be noted that while we do achieve the goal of using 
ordinary maps to represent fine shape classes, our construction is not 
category-theoretical perfect (see Remark~\ref{rem_repr}). This is one 
reason to raise a question of achieving the same result for a different 
class of spaces; another reason, of course, is to try to find a 
generalization of $|X|$ to any metrizable space $X$.

\section{Preliminaries}
\label{sec_prelim}

In this section we introduce definitions and conventions we shall use, 
as well as various previously established results on which we rely.

\subsection{Spaces}
\label{def_spaces}

We work exclusively with metrizable topological spaces and 
continuous maps between them, so we adopt the following

\begin{convention}
By a space, we shall always mean a metrizable topological space, unless 
specified otherwise. By a map between spaces, we shall always mean a 
continuous map.
\end{convention}

\begin{remark}\label{rem_metr_seq}
It is often convenient to understand the topology of a metrizable space 
in terms of point convergence; this is applicable due to the following 
property: on a given set, any two metrizable topologies with the same 
point convergence (that is, a sequence of points $\{x_{n}\}_{n \in \N}$ 
converges to a point $x$ in one topology if and only if it converges to 
the same point in the other) must coincide. Be reminded that this is not 
true if even one of these topologies is not metrizable. Similarly, a map 
$f\colon X \to Y$ between metrizable spaces is continuous if and only if 
for every sequence $\{x_{n}\}_{n \in \N}$ of points of $X$ converging to 
a point $x$, we have $f(x_{n}) \to f(x)$ in $Y$; we shall make use of 
this way to prove map continuity.
\end{remark}

\begin{notation}
We denote the homotopy class of a map $f$ by $[f]$, and we write 
$f \simeq g$ to mean that maps $f$ and $g$ are homotopic. We denote the 
set of all homotopy classes from $X$ to $Y$ by $[X,Y]$. For every space 
$X$, we denote the identity map of $X$ by $id_{X}$. Restriction of a map 
$f\colon X \to Y$ to a subspace $A \subseteq X$, we denote by $f|_{A}$.
\end{notation}

\subsection{Approaching maps}
\label{def_appr}

Following~\cite{Melikhov2022F}, we construct fine shape using 
approaching maps, which we define here.

\begin{definition}
Let $X$ be a closed subset of a space $M$. We say that $X$ is 
{\it homotopy negligible} in $M$ to mean that there exists a homotopy 
$H\colon M \times [0,1] \to M$ such that $H_{0} = id_{M}$ and 
$H(M \times (0,1]) \subseteq M\setminus X$. In other words, there is a 
deformation of $M$ into itself that never crosses $X$ except at the 
initial (identity) map.
\end{definition}

When speaking of homotopy negligible subsets, we will also use a more 
specific kind of homotopy:

\begin{definition}\label{def_H_add}
Given any homotopy $H\colon X \times [0,1] \to Y$, we say $H$ is 
{\it additive} whenever $H_{s}\circ H_{t} = H_{\min\{s+t,1\}}$ for all 
$s,t \in [0,1]$. For a space $M$ and a closed subset $X$ of $M$, we say 
that $X$ is {\it additive homotopy negligible} in $M$ whenever there is 
an additive homotopy $H\colon M \times [0,1] \to M$ such that 
$H_{0} = id_{M}$ and $H(M \times (0,1]) \subset M\setminus X$.
\end{definition}

\begin{definition}
Let $X$ and $Y$ be closed homotopy negligible in spaces $M$ and $N$ 
respectively, and let $\phi\colon M\setminus X \to N\setminus Y$ be a 
map (continuous on its domain, as per our convention). We say that 
$\phi$ is $X-Y$-{\it approaching} to mean that for any 
sequence $\{m_{i}\} \subset M\setminus X$ converging (in $M$) to a point 
of $X$, the sequence $\{\phi(m_{i})\} \subset N\setminus Y$ contains a 
subsequence that converges (in $N$) to a point of $Y$. Equivalently, 
$\phi$ being $X-Y$-approaching means that whenever a sequence in 
$M\setminus X$ has an accumulation point in $X$, the sequence's image in 
$N\setminus Y$ has an accumulation point in $Y$.
\end{definition}

Helpful equivalent definitions can be given from

\begin{proposition}(\cite[Proposition~2.5]{Melikhov2022F})\label{prop_appr_equiv}
For a map $\phi\colon M\setminus X \to N\setminus Y$, the following are 
equivalent:

(1)For every open subset $U$ of $N$, $\phi^{-1}(U\setminus Y)$ is open 
in $M$ (equivalently, $\phi^{-1}(U\setminus Y)$ is of the form 
$V\setminus X$ for some open subset $V$ of $M$);

(2)For every compact subset $C$ of $M$, $\phi(C\setminus X)$ is 
contained in a compact subset of $N$ (equivalently, 
$\phi(C\setminus X)$ is of the form $D\setminus Y$ for some compact 
subset $D$ of $N$);

(3)$\phi$ is $X-Y$-approaching.
\end{proposition}

It is clear that a composition of approaching maps is an approaching 
map, and that the identity map $id_{M\setminus X}$ is $X-X$-approaching. 

For the rest of the present work, we adopt the following

\begin{convention}
Whenever we speak of an approaching map 
$\phi\colon M\setminus X \to N\setminus Y$, we mean that $X$ and $Y$ are 
closed homotopy negligible subsets of $M$ and $N$ respectively, and 
$\phi$ is continuous (on its domain) and $X-Y$-approaching.
\end{convention}

Homotopies of approaching maps are readily defined:

\begin{definition}
Given two approaching maps 
$\phi,\psi\colon M\setminus X \to N\setminus Y$, by an {\it approaching 
homotopy} between $\phi$ and $\psi$ we mean an approaching map 
$\Phi\colon (M \times [0,1])\setminus (X \times [0,1]) \to N\setminus Y$ 
such that $\Phi_{0} = \phi$ and $\Phi_{1} = \psi$. Whenever such 
approaching homotopy exists, we say that $\phi$ and $\psi$ are 
{\it approaching homotopic}, which is an equivalence relation.
\end{definition}

\subsection{Absolute retracts}
\label{def_ARs}

The second notion we need to construct fine shape, along with 
approaching maps, is that of an absolute retract. This, of course, goes 
back to the Borsuk's definition of shape using absolute neighborhood 
retracts; here, though, we will not need the latter.

\begin{definition}
By an {\it absolute retract (AR)}, we shall mean a space $M$ that is, in 
fact, an absolute extensor for all metrizable spaces in the following 
sense: given any space $X$ and a closed subset $A$ of $X$, any map 
$f\colon A \to M$ can be extended to a map $\bar{f}\colon X \to M$ such 
that $\bar{f}|_{A} = f$:
$$\xymatrix{
A \ar[r]^{f} \ar[d]_{} & M \\
X \ar@{.>}[ur]_{\bar{f}}
}$$
\end{definition}

It is known~\cite[Corollary~18.3]{Melikhov2022T} that (at least with 
respect to metrizable spaces) absolute retracts and absolute extensors 
(AEs) are exactly the same spaces; by convention, we call them all ARs.

Aside from the definition, we make use of the following facts about ARs:

\begin{proposition}~\cite[Chapter~18]{Melikhov2022T}\label{prop_AR_props}
(1)Every contractible polyhedron is an AR, including, in particular, the 
unit interval $[0,1]$;

(2)A direct product of any countable set of ARs is an AR;

(3)For any AR $M$ and compact metrizable space $C$, the mapping space 
$M^{C}$, in the compact convergence (equivalently --- given that $C$ is 
compact --- uniform convergence) topology, is an AR;

(4)A retract of an AR is an AR;

(5)\cite[Theorem~19.3]{Melikhov2022T} If $X$ is closed and homotopy 
negligible in $M$, then $M$ is an AR if and only if $M\setminus X$ is an 
AR;

(6)For any (metrizable) space $X$, there exists an AR $M$ containing 
$X$ as a closed subset (therefore, $M \times [0,1]$ contains 
$X = X \times \{0\}$ as a closed additive homotopy negligible subset, 
homotopy given by $H_{t}(m,s) := (m,\min\{s+t,1\})$).
\end{proposition}

\subsection{Fine shape}
\label{def_fine_shape}

The notions of an approaching map and of an absolute retract are 
combined into the notion of fine shape based on 
the following fact~\cite[Lemma~2.13]{Melikhov2022F}:

\begin{lemma}\label{extend_map_to_fSh}
Let $X$ be a closed subset of any space $M$, and $Y$ be closed homotopy 
negligible in an AR $N$. Then

(1)Every map $f\colon X \to Y$ extends to a map $\bar{f}\colon M \to N$ 
such that $\bar{f}^{-1}(Y) = X$ (and therefore 
$\bar{f}|_{M\setminus X}$ is $X-Y$-approaching);

(2)For any homotopy $F\colon X \times [0,1] \to Y$ and any extensions 
$\bar{F_{0}},\bar{F_{1}}\colon M \to N$ of $F_{0}$ and $F_{1}$ such that 
$\bar{F_{0}}^{-1}(Y) = \bar{F_{1}}^{-1}(Y) = X$, there is an extension 
$\bar{F}\colon M \times [0,1] \to N$ such that 
$\bar{F}(-,0) = \bar{F_{0}}$, $\bar{F}(-,1) = \bar{F_{1}}$, and 
$\bar{F}^{-1}(Y) = X$.
\end{lemma}

\begin{proof}
(1)First we can extend $f$ to any map $f'\colon M \to N$, since $N$ is 
an AR. Then choose any homotopy $H\colon N \times [0,1] \to N$ such that 
$H_{0} = id_{N}$ and $H(N \times (0,1]) \subseteq N\setminus Y$, and any 
continuous function $h\colon M \to [0,1]$ such that $h^{-1}(0) = X$. 
From those we define $\bar{f}(m) := H_{h(m)}\circ f'(m)$. Then 
$\bar{f}(M\setminus X) \subseteq N\setminus Y$, and 
$\bar{f}|_{M\setminus X}$ is $X-Y$-approaching (as it actually extends 
on $X$ by a map into $Y$), as needed.

(2)$M \times \{0,1\} \cup X \times [0,1]$ is a closed subset of 
$M \times [0,1]$; we combine the maps $\bar{F}_{0}$, $\bar{F}_{1}$, and 
$F$ into a map of $M \times \{0,1\} \cup X \times [0,1]$ into $N$, then 
extend it to a map $\bar{F}'\colon M \times [0,1] \to N$. Now same as in 
1), choose any homotopy $H\colon N \times [0,1] \to N$ such that 
$H_{0} = id_{N}$ and $H(N \times (0,1]) \subseteq N\setminus Y$, and any 
continuous function $h\colon M \times [0,1] \to [0,1]$ such that 
$h^{-1}(0) = M \times \{0,1\} \cup X \times [0,1]$, and define 
$\bar{F}(m,t) := H_{h(m)}\circ\bar{F}'(m,t)$.
\end{proof}

Now fine shape is constructed from the following, which is clearly an 
equivalence relation:

\begin{definition}\label{def_fSh_class}
Let $X$ and $Y$ be any two spaces. Let $M$ and $M'$ be ARs, each 
containing $X$ as a closed homotopy negligible subset, whereas $N$ and 
$N'$ are ARs each containing $Y$ as such. For any approaching maps 
$\phi\colon M\setminus X \to N\setminus Y$ and 
$\psi\colon M'\setminus X \to N'\setminus Y$, we say that $\phi$ and 
$\psi$ {\it are of the same fine shape class} (from $X$ to $Y$) 
whenever there are some maps $\bar{id}_{X}\colon M \to M'$ and 
$\bar{id}_{Y}\colon N \to N'$, extending $id_{X}$ and $id_{Y}$, such 
that $\bar{id}_{X}^{-1}(X) = X$, $\bar{id}_{Y}^{-1}(Y) = Y$, and that 
there is an approaching homotopy (from $(M\setminus X) \times [0,1]$ to 
$N'\setminus Y$) between $\bar{id}_{Y}\circ\phi$ and 
$\psi\circ\bar{id}_{X}$, so that the following diagram commutes in 
approaching homotopy:
$$\xymatrix{
M\setminus X \ar[r]^{\phi} \ar[d]_{\bar{id}_{X}|_{M\setminus X}} & N\setminus Y \ar[d]^{\bar{id}_{Y}|_{N\setminus Y}} \\
M'\setminus X \ar[r]_{\psi} & N'\setminus Y
}$$
\end{definition}

By using Lemma~\ref{extend_map_to_fSh}(1), we can construct fine shape 
classes from ordinary maps:

\begin{definition}
Given a map $f\colon X \to Y$, the fine shape class (from $X$ to $Y$) 
{\it induced by $f$}, denoted $[f]_{fSh}$, is defined as follows: take 
any ARs $M$ and $N$ containing $X$ and $Y$ respectively as closed 
homotopy negligible subsets, extend $f$ to a map 
$\bar{f}\colon M\setminus X \to N\setminus Y$ such that 
$\bar{f}^{-1}(Y) = X$, and take the fine shape class of 
$\bar{f}|_{M\setminus X}$.
\end{definition}

There are, however, fine shape classes that are not induced by any 
ordinary maps.

By definition, every fine shape from $X$ to $Y$ is represented 
by some approaching map $\phi\colon M\setminus X \to N\setminus Y$ for 
some ARs $M$ and $N$ (containing $X$ and $Y$ respectively as closed 
homotopy negligible subsets). In fact, however, it can be represented 
for {\it any} such choice of ARs:

\begin{lemma}
Let $X$ and $Y$ be any two spaces, $M$ and $N$ be ARs containing $X$ and 
$Y$ respectively as closed homotopy negligible subsets, and 
$\phi\colon M\setminus X \to N\setminus Y$ be an approaching map. For 
any two other ARs $M'$ and $N'$ containing $X$ and $Y$ respectively as 
closed homotopy negligible subsets, there is an approaching map 
$\psi\colon M'\setminus X \to N'\setminus Y$ that is of the same fine 
shape class as $\phi$.
\end{lemma}

\begin{proof}
Take some maps $\bar{id}_{X}\colon M' \to M$ and 
$\bar{id}_{Y}\colon N \to N'$ extending $id_{X}$ and $id_{Y}$ 
respectively and such that $\bar{id}_{X}^{-1}(X) = X$ and 
$\bar{id}_{Y}^{-1}(Y) = Y$. Then the map 
$\bar{id}_{Y}\circ\phi\circ\bar{id}_{X}\colon M'\setminus X \to N'\setminus Y$ 
is of the same fine shape class as $\phi$.
\end{proof}

\begin{corollary}\label{fSh_composable}
Fine shape classes are composable: a fine shape class from $X$ to $Y$, 
represented by $\phi\colon M\setminus X \to N\setminus Y$, and a fine 
shape class from $Y$ to $Z$, represented by 
$\psi\colon N'\setminus Y \to L'\setminus Z$, compose through taking any 
$\phi'\colon M\setminus X \to N'\setminus Y$ of the same fine shape 
class as $\phi$ and taking the fine shape class of $\psi\circ\phi'$, or 
by taking any map $\psi'\colon N\setminus Y \to L'\setminus Z$ of the 
same fine shape class as $\psi$ and taking the fine shape class of 
$\psi'\circ\phi$ --- the two compositions are of the same fine shape 
class from $X$ to $Z$.
\end{corollary}

Thus fine shape can be defined from $X$ to $Y$ without relying on any 
specific choice of spaces containing them; this is what differentiates 
fine shape from approaching maps.

By calling onto Lemma~\ref{extend_map_to_fSh} again, we easily see that 
fine shape is weaker than homotopy:

\begin{proposition}\label{extend_homotopy_to_fSh}
The fine shape class $[f]_{fSh}$ depends only on the homotopy class $[f]$. 
Composition of maps, or of homotopy classes, induces composition of the 
corresponding fine shape classes.
\end{proposition}

We shall use the following

\begin{notation}
For an approaching map $\phi\colon M\setminus X \to N\setminus Y$, we 
shall use $[\phi]$ to denote the fine shape class from $X$ to $Y$ 
defined by $\phi$. For an actual map $f\colon X \to Y$, we shall use 
$[f]_{fSh}$ to denote the fine shape class (again from $X$ to $Y$) 
defined by $f$, or by its homotopy class $[f]$. The set of all fine 
shape classes from $X$ to $Y$, we denote by $[X,Y]_{fSh}$.
\end{notation}

Some maps induce invertible fine shape classes:

\begin{definition}\label{def_fSh_equiv}
A map $f\colon X \to Y$ is called a {\it fine shape equivalence} 
whenever $[f]_{fSh}$ has an inverse in fine shape --- that is, whenever 
there exists a (necessarily unique) fine shape class in $[Y,X]_{fSh}$, 
denoted by $[f]_{fSh}^{-1}$, such that 
$[f]_{fSh}[f]_{fSh}^{-1} = [id_{Y}]_{fSh}$ and 
$[f]_{fSh}^{-1}[f]_{fSh} = [id_{X}]_{fSh}$.
\end{definition}

\subsection{FDR-embeddings and fine shape cylinders}

We shall make use of several constructions that have been introduced for 
fine shape in~\cite{Zemlyanoy1}. All material in this subsection is 
quoted directly from there:

\begin{definition}\label{define_fdr}(\cite[Definition~3.1]{Zemlyanoy1})
Let $A$ be a closed subset of a space $X$, and assume that:

\begin{itemize}
  \item there exists an AR $M$ containing $X$ as a closed homotopy 
  negligible subset;
  \item there exists a closed subset $L$ of $M$ such that $L$ is an AR, 
  $L \cap X = A$, and $A$ is homotopy negligible in $L$;
  \item there exists an $(X \times [0,1])-X$-approaching map 
  $\Phi \colon (M \setminus X) \times [0,1] \to M \setminus X$ such that 
  $\Phi_{0} = id_{M \setminus X}$, 
  $\Phi_{1}(M \setminus X) = L \setminus A$, and 
  $\Phi_{t}|_{L \setminus A} = id_{L \setminus A}$ for all 
  $t \in [0,1]$.
\end{itemize}

In this case we say that the fine shape class $[\Phi]$ is a 
{\it fine shape strong deformation retraction} of $X$ on $A$. We also 
say that $A$ is a {\it fine shape strong deformation retract} of $X$, 
the inclusion $A \subseteq X$ is an {\it FDR-embedding}, and $\Phi$ is 
an {\it approaching strong deformation retraction} representing $[\Phi]$.
\end{definition}

In the notation of the defnition, $\Phi_{1}$ can be considered as an 
approaching map from $M\setminus X$ to $L\setminus A$, and then 
$[\Phi_{1}] \in [X,A]_{fSh}$ is the inverse to the fine shape class of 
the embedding $A \subseteq X$.

Moreover, if $A \subseteq X$ is an FDR-embedding, we can find some 
choice of $M$ and $\Phi$ for any particular choice of $L$:

\begin{lemma}\label{lem_fdr_change}(\cite[Lemma~3.7]{Zemlyanoy1})
Assume the inclusion $A \subseteq X$ is an FDR-embedding, with $M$, $L$, 
and $\Phi$ as in Definition~\ref{define_fdr}. For any other AR $L'$ 
containing $A$ as a closed homotopy negligible subset, we can always 
construct an AR $M'$ containing $X$ and $L'$ as closed subsets with 
$L' \cap X = A$ and $X$ homotopy negligible in $M'$, along with an 
approaching strong deformation retraction $\Phi'$ of $M'\setminus X$ 
onto $L'\setminus A$.
\end{lemma}

FDR-embeddings can be characterized in a simple way:

\begin{proposition}\label{prop_fdr_char}(\cite[Theorem~3.13]{Zemlyanoy1})
A map is an FDR-embedding if and only if it is a closed embedding and a 
fine shape equivalence (see Definition~\ref{def_fSh_equiv}).
\end{proposition}

One property we will want to use is given by

\begin{proposition}\label{prop_II_fdr}(\cite[Corollary~3.11]{Zemlyanoy1})
If $A \subseteq W$ is an FDR-embedding, then so is 
$A \times [0,1] \cup W \times \{0,1\} \subseteq W \times [0,1]$.
\end{proposition}

Finally, we condense an important result~\cite[Section~6]{Zemlyanoy1} in 
the following form suitable for our purposes:

\begin{theorem}\label{th_fc}
For any fine shape class $[\phi] \in [X,Y]_{fSh}$, there exists a space 
$F$ (called the fine shape cylinder of $[\phi]$), along with closed 
embeddings $u\colon X \to F$ and $i\colon Y \to F$, such that $i$ is an 
FDR-embedding, and $[i]_{fSh}^{-1}[u]_{fSh} = [\phi]$. Moreover, such a 
space $F$ is unique up to a fine shape equivalence.
\end{theorem}

Note that here $[i]_{fSh}^{-1}$ is well-defined, as per 
Proposition~\ref{prop_fdr_char} and Definition~\ref{def_fSh_equiv}.

\section{Representing fine shape from arbitrary metrizable spaces into a compactum}
\label{sec_compact_case}

As it turns out, we want to have the intended construction for compact 
metrizable spaces, so that we can reference it in defining the more 
general one; thus this section shall be devoted to describing the more 
specific case. Fortunately, Cathey's work~\cite{Cathey1981} already 
gives the whole construction, although stated in terms we will not need; 
moreover, Cathey there is almost entirely focused on compact spaces, and 
the resulting statements will not be directly sufficient for our 
purposes. As such, here we provide what is effectively the same 
construction, but restated and explored the way we shall use it --- both 
removing the references to other terms that Cathey needed but we do not, 
and expanding the properties to include the non-compact spaces.

One difference we should mention explicitly is that Cathey's definition 
of an SSDR-map (Definition~(1.1) of~\cite{Cathey1981}), given for 
compacta, has been subsumed by our definition of an FDR-embedding of 
arbitrary metrizable spaces (see~\ref{define_fdr}). 
With this in mind, we can explain 
our intent: for a compact metrizable space $X$, Cathey defines a 
metrizable (but generally noncompact) space $|X|$ (Theorem~(2.5) there) 
that has two properties. One is that $X$ is FDR-embedded in $|X|$, 
albeit Cathey does not use that terminology --- nor is that called an 
SSDR-map, as it was not defined for cases where $|X|$ is not compact; 
we, of course, shall simply state it to be an FDR-embedding. The second 
property is that for any SSDR-map $i\colon A \to W$ and any map 
$f\colon A \to |X|$, there is a map $f'\colon W \to |X|$ such that 
$f'i = f$ --- that is, any map into $|X|$ extends onto any space 
containing its original domain via an SSDR-map (Cathey reasonably calls 
such spaces ``fibrant'', using that name --- Definition~(2.1) --- for 
all metrizable spaces rather than only for compact ones). For our 
purposes, this is insufficient strictly as stated: we want to use the 
same property, but with FDR-embeddings instead of only SSDR-maps, 
meaning spaces $A$ and $W$ may be noncompact, and that is precisely the 
way we state it after we define the space $|X|$ in suitable terms. We 
also provide the proofs for everything we state; while most of those 
could in principle be inferred from Cathey's work, the wide difference 
in approaches between there and here would be quite demanding. As such, 
we greatly prefer to have all proofs in the form compatible with our 
methodology.

Having stated our goal for this section, we can start with the 
constructions we shall use:

\begin{definition}\label{def_metr_quot_comp}
Let $M$ be a metrizable topological space, $X$ a compact subset of $M$, 
and $d$ a metric giving the topology of $M$. We denote by $d_{M/X}$ the 
metric on the quotient space $M/X$ defined by 
$d_{M/X}(m,\{X\}) := d(m,X)$ and 
$d_{M/X}(m,m') := \min\{d(m,m'),d(m,\{X\}) + d(m'\{X\})\}$ for all 
$m,m' \in M$.
\end{definition}

\begin{remark}\label{rem_not_quotient}
Clearly $d_{M/X}$ is a metric, and gives the quotient topology on $M/X$. 
The latter is not true if $X$ is not compact; we return to that in 
Definition~\ref{def_metr_quot_gen}. Also, the set $M\setminus X$ has the 
same topology whether considered as the subset of $M$ or the subset of $M/X$.
\end{remark}

\begin{definition}(cf.~\cite[Theorem~(2.5)]{Cathey1981})\label{def_compact_class}
Let $X$ be a compact metrizable space, and choose a compact AR $M$ 
containing $X$ as a closed additive homotopy negligible subset. In the 
path space $(M/X)^{[0,1]}$, denote by $|M|$ the subset of all paths 
$\gamma\colon [0,1] \to M/X$ such that 
$\gamma((0,1]) \subseteq M\setminus X$. Also denote by $|X|$ the 
subset of $|M|$ consisting of all $\gamma$ with $\gamma(0) = \{X\}$.
\end{definition}

\begin{remark}\label{rem_on_class_topology}
(1)Note how a map $\phi\colon [0,1]\setminus \{0\} \to M\setminus X$ is 
an approaching map if and only if the obvious corresponding map from 
$[0,1]$ to $M/X$ is continuous, therefore a path in $M/X$, thus a point 
of $|X|$;

(2)In accordance with Remark~\ref{rem_metr_seq}, it is convenient to 
understand the 
topology of $|M|$ in terms of point convergence: $|M|\setminus |X|$ is 
homeomorphic to the path space of $M\setminus X$, and has the same point 
convergence; and for a point $\gamma \in |X|$, a sequence $\{\gamma_{n}\}$ 
of points of $|M|$ converges to $\gamma$ if and only if (a)for every 
$s \in (0,1]$, every open neighborhood of $\gamma(s)$ in $M$ 
(or even in $M\setminus X$) contains $\gamma_{n}([s-\sigma,s+\sigma])$ 
for some $\sigma > 0$ and all $n$ large enough, and (b)every open 
neighborhood of $X$ in $M$ contains $\gamma_{n}((0,\sigma])$ for some 
$\sigma > 0$ and all $n$ large enough. Be warned that condition~(b) 
cannot be replaced by requiring that $\gamma_{n}(0) \to \{X\}$, 
$n \to \infty$, in $M/X$.

(3)Combining (1) and (2), we can see that a path 
$\Gamma\colon [0,1] \to |X|$ is precisely an approaching map 
$\Gamma'\colon ([0,1] \times [0,1])\setminus ([0,1] \times \{0\}) \to M\setminus X$, 
where $[\Gamma(t)](s) = \Gamma'(t,s)$.
\end{remark}

\begin{proposition}\label{prop_compact_classifier}
In the notation of the preceding definition,

(1)$|M|$ is an AR containing $|X|$ as a closed additive homotopy negligible subset;

(2)There is an FDR-embedding of $X$ in $|X|$, with an approaching strong 
deformation retraction of $|M|\setminus |X|$ onto $M\setminus X$;

(3)For any FDR-embedding $i\colon A \to W$ and any map $f\colon A \to |X|$, 
there is a map $f'\colon W \to |X|$ extending $f$ (that is, $f'i = f$).
\end{proposition}

\begin{remark}\label{rem_compact_univ}
By the universal property given by (2) and (3), for any space $X$, 
$|X|$ is unique up to a homotopy equivalence that is constant on $X$; we 
elaborate on this for the general case in Corollary~\ref{cor_class}(2).
\end{remark}

\begin{proof}[\it Proof (of Proposition~\ref{prop_compact_classifier})]
(1)Clearly $|X|$ is closed in $|M|$, and also additive homotopy 
negligible via the homotopy $\Gamma\colon |M| \times [0,1] \to |M|$ 
defined by $[\Gamma(\gamma,s)](t) := \gamma(\max\{s,t\})$. Now 
$|M|\setminus |X|$ is just $(M\setminus X)^{[0,1]}$. But $M\setminus X$ 
is an AR by Proposition~\ref{prop_AR_props}(5), and then by 
Proposition~\ref{prop_AR_props}(3) so is the space of paths in it; 
therefore $|M|\setminus |X|$ is an AR, and by 
Proposition~\ref{prop_AR_props}(5) again so is $|M|$.

(2)Let $H\colon M \times [0,1] \to M$ be any additive homotopy with 
$H_{0} = id_{M}$ and $H(M \times (0,1]) \subseteq M\setminus X$. Embed 
$X$ into $|X|$ by sending $x \in X$ to 
$\gamma_{x}\colon [0,1] \to M/X$ for which 
$\gamma_{x}(t) := H(x,t) \in M\setminus X$ for $t \in (0,1]$, and 
$\gamma_{x}(0) := \{X\}$ (note that this embedding is injective due to 
Hausdorf property of $M$). Extend this to the embedding of 
$M$ into $|M|$ by $\gamma_{m}(t) := H(m,t)$ for $m \in M\setminus X$. 
Now the approaching strong deformation retraction 
$R\colon (|M|\setminus |X|) \times [0,1] \to |M|\setminus |X|$ can be 
defined by 
$$[R(\gamma,s)](t) := \begin{cases} 
  \gamma(t),\,t \leq s; \\
  H(\gamma(s),t-s),\,t > s.
\end{cases}
$$
$R$ is continuous, and is constant on the copy of $M\setminus X$ in 
$|M|\setminus |X|$ because $H$ was chosen additive; now assume a 
sequence $\{\gamma_{n}\}$ in $|M|\setminus |X|$ converges to a point 
$\gamma \in |X|$, and also $s_{n} \to s$. For $s \neq 0$, 
$\{R(\gamma_{n},s_{n})\}$ simply converges to the path 
$\gamma^{(s)} \in |X|$ defined by 
$$\gamma^{(s)}(t) := \begin{cases} 
  \gamma(t),\,t \leq s; \\
  H(\gamma(s),t-s),\,t > s
\end{cases}
$$
\noindent (in other words, $\Gamma$ extends continuously onto 
$|X| \times (0,1]$). 
For $s = 0$, the points $\gamma_{n}(s_{n})$, converge to $X$ 
in $M$ (and then to $\{X\}$ in $M/X$), and thus have an accumulation 
point $x \in X$ in it by compactness (this is where the construction 
fails if $X$ is noncompact). Therefore also the paths 
$R(\gamma_{n},s_{n})$ have the path $\gamma_{x}$ as an accumulation 
point in $|M|$; note that the sequence $\{R(\gamma_{n},s_{n})\}$ 
satisfies condition~(b) of Remark~\ref{rem_on_class_topology}(2) because 
the sequence $\{\gamma_{n}\}$ must do so to converge to a point of $|X|$.

(3)Assume we are given a map $f\colon A \to |X|$ and an FDR-embedding 
$i\colon A \to W$. We need to extend $f$ to a map $f'\colon W \to |X|$. 
Let $L'$ be any AR containing $A$ as a closed homotopy negligible subset 
and take $L := L' \times [0,1]$, an AR containing $A = A \times \{0\}$ 
in the same way, with homotopy $H\colon L \times [0,1] \to L$ defined by 
$H((l,t),s) := (l,\max\{t,s\})$, where $(l,t) \in L' \times [0,1] = L$. 
Now $f$ extends to some map $F\colon L \to |M|$ such that 
$F^{-1}(|X|) = A$ and also for all $(a,s) \in A \times [0,1] \subset L$, 
$F(a,s) = F(H((a,s)) = [f(a)](s)$ (note that replacing $L'$ 
by $L' \times [0,1]$ was done to ensure that $H$ is injective on 
$A \times [0,1]$). Using Lemma~\ref{lem_fdr_change}, we can choose some 
$N$, $\bar{H}$, and 
$R'$, where $N$ is an AR containing $W$ and $L$ as closed subsets, with 
$L \cap W = A$, $\bar{H} \colon N \times [0,1] \to N$ is a homotopy 
extending $H$ in such way that $\bar{H}_{0} = id_{N}$ and 
$\bar{H}(N \times (0,1]) \subseteq N\setminus W$, and 
$R'\colon (N\setminus W) \times [0,1] \to N\setminus W$ is an 
approaching strong deformation retraction of $N\setminus W$ onto 
$L\setminus A$. Then we define $f'\colon W \to |X|$ via 
$[f'(p)](t) := RFR'H(p,t)$, $t \in (0,1]$; now $f'$ restricts to $f$ 
by the construction of $F$, and by $R'$ and $R$ being constant on 
$L\setminus A$ and $M\setminus X$ respectively. Moreover, $f'$ is a 
continuous map, and it is 
important to understand why: by Remark~\ref{rem_metr_seq}, assume we 
have $w_{n} \to w$ in $W$, and see how $f'(w_{n}) \to f'(w)$. The 
condition~(a) of Remark~\ref{rem_on_class_topology}(2) is satisfied 
because $H$, $R'$, $F$, and $R$ are all continuous on their respective 
domains; for the condition~(b), note that if we take an open 
neighborhood $U$ of $X$ in $M$, then the preimage 
$(RFR')^{-1}(U \cap (M\setminus X)) \subseteq N\setminus W$ must contain 
$V \cap (N\setminus W)$ for some open neighborhood $V$ of $w$ in $N$ 
(otherwise, there is a sequence $v_{n}$ of points of $N\setminus W$ such 
that $v_{n} \to w$, but $\{RFR'(v_{n})\}_{n \in \N}$ has no accumulation 
points in $X$, contradicting the fact that $R$, $F$, and $R'$ are all 
approaching); but by $w_{n} \to w$ and $H$ being continuous, $V$ contains 
$H(\{w_{n}\} \times [0,\sigma])$ for some $\sigma > 0$ and all $n$ large 
enough, thus $U$ also contains $[f'(w_{n})]((0,\sigma])$ for all such $n$.
\end{proof}

\begin{remark}
For noncompact $X$, the construction of $|X|$ fails in multiple ways. 
The most obvious problem is that $M/X$ is not metrizable, but that can 
be redefined (see Definition~\ref{def_metr_quot_gen}). Another, 
mentioned in the preceding proof, is that the retraction $R_{1}$ is not 
$|X|-X$-approaching anymore: there 
must exist a sequence $\{m_{n}\}$ of points of $M$ converging to $X$ yet 
without any accumulation points in it (or indeed in all of $M$). Choose 
a sequence of open neighborhoods $U_{n}$ of $X$ in $M$ such that 
$m_{n} \in U_{n}$ and $U_{n+1} \subseteq U_{n}$ for all $n$, and 
moreover $\bigcap_{n \in \N}\, U_{n} = X$ (we can do so because $X$ is a 
closed subset of a metrizable --- therefore perfectly normal --- space). 
Now take a path $\gamma \in |X|$, and define a sequence of paths 
$\{\gamma_{n}\}$ in $|M|\setminus |X|$ so that 
$\gamma_{n}|_{[1/n,1]} = \gamma|_{[1/n,1]}$ and 
$\gamma_{n}([0,1/n]) \subseteq U_{n}$ for all $n$ (and therefore 
$\gamma_{n} \to \gamma$ in $|X$|), but $\gamma_{n}(0) = m_{n}$; then the 
sequence $\{R(\gamma_{n},1)\}$ does not have an accumulation point in 
the copy of $X$ in $|X|$. It is easy, using 
Definition~\ref{def_metr_quot_gen}, to construct such an example with 
$M = (0,1) \times [0,1]$, $X = (0,1) \times \{0\}$, $m_{n} = (1/n,1/n)$, 
$\gamma(s) = (1/2,s)$ ($\gamma(0) = \{X\}$, in the {\it set} $M/X$ with 
the defined metrizable topology), and 
$\gamma_{n}|_{[0,1/n]}(s) = (1/n + s(1/2 - 1/n),1/n)$, assuming the 
metric is chosen in the obvious way so that we can take 
$U_{n} = (0,1) \times [0,1/n)$. There is, in fact, yet another problem: 
it might not be possible to choose a metric topology for $M/X$ such that 
$|X|$ is homotopy negligible in $|M|$.
%$[]$

Avoiding all of these obstacles at once seems to require 
not only some changes in the construction, but also placing additional 
restrictions on $X$; this is what Section~\ref{sec_classifier} explores.

One more consideration is that $|X|$ is not compact despite $X$ 
being so; this is to be expected, as $|X|$ is effectively the strong 
shape version of a space of paths (from one-point space to $X$). Due to 
this, Cathey had to use this construction in a roundabout way, not being 
able to utilize the whole space $|X|$ directly, and this is also what 
prevents building the model structure while working with compact spaces 
only (because model structure implies existence of path spaces). We, 
however, seeking specifically to work also with noncompact spaces, will 
not be hindered by $|X|$ not being compact in the present work.
\end{remark}
\vskip 1em
Thus by the end of this section, we have the space $|X|$ for every 
compact metrizable space $X$ defined and studied in terms we need. Our 
goal from now on is to give an equivalent construction for all locally 
compact metrizable spaces; for this, we will first explore a particular 
class of function on such spaces that will be used in our construction.

As a note, we did not yet actually show that the space $|X|$ is in a 
sense unique for any given $X$; we shall prove this in the general case.

\section{Exhaustion functions}
\label{sec_exh_fun}

This section describes the construction of functions of a specific kind 
on metrizable topological spaces. The statements and proofs given here 
are easy to understand, but still necessary for the following section; 
thus, we provide them in full for clarity.

We begin by stating a standard form of some definitions we shall use:

\begin{definition}
For a topological space $X$, we say that $X$ is locally compact whenever 
every point of $X$ has an open neighborhood with compact closure; we say 
that $X$ is separable whenever $X$ is the closure of a countable subset 
of $X$; by a local compactum, we mean a locally compact separable 
metrizable space; and if $X$ is locally compact and Hausdorf, we use 
$X^{*} := X \cup {\infty}$ to denote the one-point closure of $X$.
\end{definition}

Now we define what we shall call an exhaustion function:

\begin{definition}
Let $X$ be a metrizable topological space. A continuous function 
$u\colon X \to (0,1]$ will be called an {\it exhaustion function} on $X$ 
if $u(x_n) \to 0,\,n \to \infty$, for every sequence ${x_n}$ of points 
of $X$ that has no accumulation points in $X$.
\end{definition}

Clearly, such a function $u\colon X \to (0,1]$ gives an exhaustion of 
$X$ by the non-decreasing sequence of compact sets 
$u^{-1}([\frac{1}{n},1])$, $n = 1,\,2,\,\ldots$. In fact, such a 
function is effectively a continuous equivalent of an exhaustion by 
sequence, assigning a compact set $u^{-1}([\varepsilon,1])$ to every 
$\varepsilon \in (0,1]$; this, of course, is the reason for the name. 
For completeness, we give a thorough description of when exhaustion functions exist:

\begin{proposition}
\label{prop_exh_equiv}
For a noncompact metrizable space $X$, the following are equivalent:

(1)There exists an exhaustion function on $X$;

(2)$X$ is the direct limit of a sequence of its compact subspaces 
$X_{0} \subseteq X_{1} \ldots \subseteq X_{n} \subseteq \ldots$ with the 
inclusion maps (and therefore, $X = \bigcup_{n \in \N} X_{n}$) such that 
$X_{n}$ is contained in the interior (in $X$) of $X_{n+1}$ for all $n$;

(3)$X$ is locally compact and separable (that is, a local compactum);

(4)$X$ is locally compact, and the point $\infty$ of $X^{*}$ has a countable 
basis of open neighborhoods;

(5)$X$ is locally compact, and $X^{*}$ is first-countable;

(6)$X$ is locally compact, and $X^{*}$ is perfectly normal;

(7)$X$ is locally compact, and $X^{*}$ is metrizable.
\end{proposition}

\begin{proof}
$(1) \Rightarrow (2)$ Given an exhaustion function $u$, take 
$X_{n} := u^{-1}([1/n,1])$; it is compact from the defining property of 
an exhaustion function. Moreover, $u^{-1}((1/(n+1),1])$ is open, so
$X_{n} \subseteq u^{-1}((1/(n+1),1]) \subseteq X_{n+1}$ shows that 
$X_{n}$ is contained in the interior of $X_{n+1}$.

$(2) \Rightarrow (3)$ Each $X_n$ is separable, being compact and 
metrizable; thus $X$, a countable union of all $X_n$, is also separable. 
Now every point of $x$ is contained in $X_{n}$ for some $n$, and thus 
has an open neighborhood $U$ contained in the interior of $X_{n+1}$; the 
closure of $U$ is then compact (contained in $X_{n+1}$, in fact).

$(3) \Rightarrow (4)$ Let $Y$ be a countable dense subset of $X$. Call 
an open subset $U$ of $X$ {\it marked} whenever if satisfies any (and 
then every) of the following equivalent conditions:

- $U$ has compact closure in $X$;

- the closure of $U$ in $X^{*}$ does not contain $\infty$;

- $\infty$ has an open neighborhood in $X^{*}$ disjoint from $U$.

Every point of $X$ has a marked neighborhood, and every open subset of a 
marked set is marked; choose a metric $d$ giving the topology of $X$, 
and let $F$ be the collection of all marked open $d$-balls $B(y,1/n)$ of 
radius $1/n$, $n = 1,\,2,\,\ldots$, with center 
$y \in Y$. Then $F$ is countable, and the sets of $F$ cover $X$: for a 
point $x \in X$, some open ball $B(x,2\varepsilon)$ has compact closure 
(in $X$), and $B(x,\varepsilon)$ contains some $y \in Y$ (as $Y$ is 
dense in $X$), so $B(y,\varepsilon)$ is marked (thus in $F$) and 
contains $x$. Now enumerate the sets in $F$ in some way as 
$U_{0},\, U_{1},\, \ldots,\, U_{n},\, \ldots$; for every $n \in \N$, 
take $V_{n}$ to be an open neighborhood of $\infty$ in $X^{*}$ disjoint 
from $U_{n}$, and define $W_{n} := \bigcap_{k = 0}^{n} V_{k}$. Then $W_{n}$ 
are open neighborhoods of $\infty$ with $W_{n+1} \subseteq W_{n}$ for 
all $n$. To see that $W_{n}$ form a basis of the system of open 
neighborhoods of $\infty$, take an arbitrary open neighborhood $V$ of 
$\infty$; the compact set $X^{*}\setminus V$ is covered by a finite 
subcollection of $F$, say by the sets $U_{0},\,\ldots,\,U_{n}$. Then 
$W_{n}$ is a open neighborhood of $\infty$ disjoint from each of these, 
therefore disjoint from $X^{*}\setminus V$, thus contained in $V$.

$(4) \Rightarrow (5)$ $X$, being metrizable, is first countable, and 
thus every point $x \in X$ also has a basis of the system of its open 
neighborhoods in $X^{*}$, by first restricting to an open neighborhood with 
compact closure in $X$. Since $\infty$ also has such a basis by the 
assumption, this makes $X^{*}$ first countable.

$(5) \Rightarrow (6)$ Let $\{W_{n}\}$ be a countable basis of the 
system of open neighborhoods of $\infty$ in $X^{*}$; then 
$\bigcap_{n \in \N} W_{n} = \{\infty\}$ (as $X^{*}$ is Hausdorf). Given a 
closed set $Z$ in $X^{*}$, $Z\setminus \{\infty\}$ is closed in $X$, thus 
(as $X$ is metrizable and so perfectly normal) an intersection of some 
countable collection of sets $U_{n}$ that are open in $X$ --- and thus 
also in $X^{*}$. Now if $\infty \notin Z$ or, equivalently, 
$Z = Z\setminus \{\infty\}$, then $Z = \bigcap_{n \in \N} U_{n}$ also 
proves $Z$ to be a $\mathrm{G}_{\delta}$ set in $X^{*}$; otherwise, this is 
proved by 
$Z = (Z\setminus \{\infty\}) \cup \{\infty\} = \bigcap_{n \in \N} (U_{n} \cup W_{n})$.

$(6) \Rightarrow (7)$ It is a well-known consequence of the Urysohn's 
construction that for any normal topological space $X$ and any closed
$\mathrm{G}_{\delta}$ subset $Z$ of $X$, there is a continuous function 
$f\colon X \to [0,1]$ such that $Z = f^{-1}(0)$ (whereas for $Z$ not 
$\mathrm{G}_{\delta}$, we can only ensure that $Z \subseteq f^{-1}(0)$). 
Thus let $u\colon X* \to [0,1]$ be continuous with 
$u^{-1}(0) = \{\infty\}$. Now $d^{u}(x,x') := |u(x) - u(x')|$ is a 
pseudometric on $X$; if the topology of $X$ is given by a metric $d$, 
then $d + d^{u}$ is also a metric giving the same topology. Define a 
metric $d^{*}$ on $X^{*}$ by $d^{*}(\infty,x) := u(x)$ and 
$d^{*}(x,x') := \min\{u(x) + u(x'),d(x,x') + |u(x) - u(x')|\}$ for all 
$x,x' \in X$. After checking that the triangle inequality holds, we see 
that $d^{*}$ is in fact a metric, and it gives the topology of $X^{*}$: the 
$d^{*}$-open balls $B_{d^{*}}(\infty,c)$ are open in $X^{*}$ as their complements 
$u^{-1}([c,1])$ are compact for $c > 0$, and form the basis of the 
system of open neighborhoods of $\infty$, since for any open 
neighborhood $V$ of $\infty$, $u$ has a nonzero minimum $c$ on the 
compact set $X^{*}\setminus V$, so $V$ contains $B_{d^{*}}(\infty,c)$.

$(7) \Rightarrow (1)$ For any metric $d^{*}$ providing the topology of 
$X^{*}$, $d^{*}(\infty,-)\colon X \ni x \mapsto d*(\infty,x) \in (0,1]$ is an 
exhaustion function on $X$.
\end{proof}

\begin{remark}
As we can see now, all local compacta --- and only local compacta --- 
have exhaustion functions. For a compactum, the constant function equal 
to $1$ everywhere is an exhaustion function; for a non-compact local 
compactum, the image of any exhaustion function must contain, more or 
less by definition, values arbitrarily close to $0$.
\end{remark}

\section{Representing fine shape from locally compact spaces into a local compactum}
\label{sec_classifier}

Now we are in position to achieve our aim: while 
Section~\ref{sec_compact_case} was effectively a restatement of 
Cathey's relevant results of~\cite{Cathey1981} in terms we use, here we 
shall extend the same to the case of non-compact spaces. Our approach 
relies on exhaustion functions, which exist precisely on locally compact 
separable metrizable spaces (local compacta), and therefore these will 
constitute the class of spaces we study.

We proceed in the same way as in Section~\ref{sec_compact_case}, 
modifying our expanded definitions according to the differences caused 
by the loss of compactness:

\begin{definition}\label{def_metr_quot_gen}
Let $X$ be a closed subset of a metrizable space $M$. For any metric $d$ 
giving the topology of $M$, define a metric $d_{M/X}$ on the set $M/X$ 
by $d_{M/X}(m,\{X\}) := d(m,X)$ and $d_{M/X}(m,m') := \min\{d(m,m'),d(m,X)+d(m',X)\}$ 
for all $m,m' \in M\setminus X$.
\end{definition}

\begin{remark}\label{rem_quot_diff}
As noted in~Remark~\ref{rem_not_quotient}, for noncompact $X$, different choices of $d$ 
will result in topologically different metrics $d_{M/X}$. However, for 
any closed subset $M'$ of $M$ with $M' \cap X$ compact, the inherited 
topology on $M'/(M' \cap X)$ is the quotient topology (although the 
metric $d_{M/X}$ restricted to $M'/(M' \cap X)$ is not the same as 
$d_{M'/(M' \cap X)}$).
\end{remark}

\begin{lemma}\label{lem_exh_AR}
For every local compactum $X$, there exist a locally compact separable 
AR $M$ that contains $X$ as a closed additive homotopy negligible 
subset, an exhaustion function $u$ on $M$ (and then $u|_{X}$ is an 
exhaustion function on $X$), and an additive homotopy 
$H\colon M \times [0,1] \to M$ such that 
$H_{0} = id_{M}$, $H(M \times (0,1]) \subseteq M\setminus X$, and also 
$u(H(m,s)) \geq u(m)$ for all $m \in M$ and all $s \in [0,1]$. In fact, 
it is always possible to have $u(H(m,s)) = u(m)$ for all $m$ and all $s$.
\end{lemma}

\begin{proof}
By~\cite[Lemma~3.18]{Melikhov2022F}, we can choose $M$ that is locally 
compact and separable, and contains $X$ as a closed subset; now for any 
exhaustion function $u$ on $M$, replace $M$ by $M \times [0,1]$ (with 
$X = X \times \{0\}$), $u$ by $(m,s) \mapsto u(m)$, and take 
$H((m,s),t) := (m,\max\{s,t\})$.
\end{proof}

\begin{definition}\label{def_classifier}
Let $X$ be a local compactum, and choose $M$, $H$, and $u$ as given by 
the previous lemma. For every $\varepsilon \in (0,1]$, define compact 
subspaces $M_{\varepsilon} := u^{-1}([\varepsilon,1])$ and 
$X_{\varepsilon} := M_{\varepsilon} \cap X$. 
Definition~\ref{def_compact_class} gives us spaces $|M_\varepsilon|$ and 
$|X_{\varepsilon}|$; any metric $d$ giving the topology of $M$ metrizes 
all these via restricting $d_{M/X}$, as per Remark~\ref{rem_quot_diff} 
and Definition~\ref{def_compact_class}; thus we have an isometric 
embedding of $|M_\varepsilon|$ in $|M_\varepsilon'|$ (as metric spaces) 
whenever 
$\varepsilon > \varepsilon'$. Finally, define $|M|$ to be the union 
$\bigcup_{\varepsilon \in (0,1]} \{\varepsilon\} \times |M_{\varepsilon}|$, 
metrizable by 
$d_{|M|}((\varepsilon,\gamma),(\varepsilon',\gamma')) := 
|\varepsilon - \varepsilon'| + d_{|M_{\min\{\varepsilon,\varepsilon'\}}|}(\gamma,\gamma')$; 
here, we rely on the isometric embeddings for the second term to be 
well-defined.
\end{definition}

\begin{remark}
For compact $X$, we can choose $M$ compact, and $u$ everywhere equal to 
$1$; then $|M|$ will be the same as $|M|$ of 
Definition~\ref{def_compact_class} multiplied by $(0,1]$. As stated 
later in Corollary~\ref{cor_class}, $|M|$ is in general unique up to a 
homotopy equivalence; thus, up to the same, 
Definition~\ref{def_classifier} restricts to 
Definition~\ref{def_compact_class} on compacta.
\end{remark}
\vskip 1em
For this new definition, we want to extend the proterties of 
Proposition~\ref{prop_compact_classifier}:

\begin{proposition}\label{pr_class_ext}
In the notation of the preceding definition, 

(1)$|M|$ is an AR containing $|X|$ as a closed additive homotopy 
negligible subset;

(2)There is an FDR-embedding of $X$ into $|X|$, extending to a closed 
embedding of $M$ into $|M|$ with $X = |X| \cap M$ and an approaching 
strong deformation retraction of $|M|\setminus |X|$ onto $M\setminus X$.
\end{proposition}

\begin{proof}
(1)As in the proof of Proposition~\ref{prop_compact_classifier}(1), for 
all $\varepsilon \in (0,1]$, 
$|M_{\varepsilon}|$ is an AR containing $|X_{\varepsilon}|$ as a closed 
additive homotopy negligible subset; the same homotopy $\Gamma$ proves 
that $|X|$ is additive homotopy negligible in $|M|$: using 
Remark~\ref{rem_metr_seq}, assume 
that $(\varepsilon_{n},\gamma_{n}) \to (\varepsilon,\gamma)$, and see 
that for some $\varepsilon' \in (0,\varepsilon]$, $|M_{\varepsilon'}|$ 
contains $\gamma_{n}$ for all large enough $n$, and then for any 
converging sequence $s_{n} \to s$ in $[0,1]$, 
$\Gamma(\gamma_{n},s_{n}) \to \Gamma(\gamma,s)$ in $|M_{\varepsilon'}|$ 
(be warned that the existence of such $\varepsilon'$ is important for 
continuity here). Now the space $(M\setminus X)^{[0,1]}$ 
of paths in $M\setminus X$ is an AR by Proposition~\ref{prop_AR_props}(3), 
and $|M|\setminus |X|$ is a retract of $(0,1] \times (M\setminus X)^{[0,1]}$ via 
$(q,\gamma) \mapsto (\min\{q,\inf_{s \in [0,1]} u(\gamma(s))\},\gamma)$; 
therefore by Proposition~\ref{prop_AR_props}(5) and~(2) and~(4), $|M|\setminus |X|$ 
is an AR, and again by~(5) so is $|M|$.

(2)Embed $M$ into $|M|$ via $m \mapsto (u(m),\gamma_{m})$, where 
$\gamma_{m}(s) := H(m,s)$; here for $m \in X$, 
$\gamma_{m}(0) := \{X_{u(m)}\} \in M_{u(m)}/X_{u(m)}$, 
and also note that $u(\gamma_{m}(s)) \geq u(m)$ for all $s \in (0,1]$, 
since $u$ and $H$ were chosen via Lemma~\ref{lem_exh_AR}. Also denote by 
$M'$ the subset $\{(q,\gamma_{m}) \mid m \in M, u(m) \geq q\}$ of $|M|$ 
and take $X' := |X| \cap M'$; clearly $M \subseteq M' \subseteq |M|$, 
and there is a strong deformation retraction of $M'$ onto $M$ that 
restricts to a strong deformation retraction of $X'$ onto $X$. Thus it 
suffices to construct an approaching strong deformation retraction of 
$|M|\setminus |X|$ onto $M'\setminus X'$.

Now, similar to Proposition~\ref{prop_compact_classifier}(2), we define 
that approaching strong deformation retraction 
$R\colon (|M|\setminus |X|) \times [0,1] \to |M|\setminus |X|$ 
by $R((q,\gamma),s) := (q,\gamma_{s})$, where
$$\gamma_{s}(t) := \begin{cases} 
  \gamma(t),\,t \leq s; \\
  H(\gamma(s),t-s),\,t > s;
\end{cases}
$$
\noindent note that for $t \in [s,1]$, $u(\gamma_{s}(t)) \geq u(\gamma(s)) \geq q$ 
due to how $u$ and $H$ were chosen (also, $u(\gamma(0))$ is always 
defined and a point of $M\setminus X$, as $\gamma \in |M|\setminus |X|$), 
so that $R_{s}(q,\gamma)$ always stays in $|M|\setminus |X|$. Then $R$ 
is continuous, the image of $R_{1}$ is precisely $M'\setminus X'$, and 
$R$ stays constant on $M'\setminus X'$ (again, because $H$ was chosen additive), 
so it remains to show that $R$ is approaching. This is shown as in 
Proposition~\ref{prop_compact_classifier}(2), with one difference: if a 
sequence $\{(q_{n},\gamma_{n})\}$ in $|M|\setminus |X|$ converges to 
$(q,\gamma) \in |X|$ and $s_{n} \to 0$ in $[0,1]$, then 
$u(\gamma_{n}(s_{n})) \geq q_{n}$, so the sequence 
$\{\gamma_{n}(s_{n})\}$ (in $M$) has an accumulation point in the 
compact subset $u^{-1}([q,1]) \cap X = X_{q}$ of $M$.
\end{proof}

\begin{remark}
We know that the topological quotient space $M/X$ may not be metrizable 
for a noncompact $X$, and that different choices of a metric $d$ on $M$ 
lead to different topologies on the {\it set} $M/X$. Even then, it would 
be possible to extend $|X|$ of Definition~\ref{def_compact_class} by 
defining it to be the space of all approaching maps 
$\gamma\colon [0,1]\setminus\{0\} \to M\setminus X$ with the compact 
convergence topology; in the compact case, this is homeomorphic to the 
path space of Definition~\ref{def_compact_class}. Even in the noncompact 
case, the proof of Proposition~\ref{prop_compact_classifier}(3) works 
without changes. Finally, $X$ clearly embeds in 
that space in much the same way as in Definition~\ref{def_classifier}. 
The last step of the preceding proof, however, shows why that embedding 
is not an FDR-embedding, thus invalidating the whole approach: the 
sequence $\gamma_{n}(s_{n})$ might not have an accumulation point in the 
noncompact space $M$, which is why we first introduce some exhaustion of 
$M$ by compact subspaces and define a topology that relies on it.
\end{remark}

At this point, we obviously want to state the equivalent of 
Proposition~\ref{prop_compact_classifier}(3) to complete what we claim 
to be an extension of the whole Section~\ref{sec_compact_case}. This is 
the primary result of the present work, and it will be different in two 
notable ways. First, we will show extension of maps only to within a 
homotopy; we believe this to be reasonable, because the space $|X|$, as 
we are about to show, is unique (for a given $X$) up to a homotopy 
equivalence --- not up to a homeomorphism, not even for a compact $X$. 
As to the second difference, the proof will, of course, require 
the existence of exhaustion functions on $X$, meaning $X$ must be a 
local compactum --- a separable locally compact metrizable space. The 
spaces $A$ and $W$ (as in the compact version), on the other hand, need 
not be separable, and yet it is telling that the following proof does 
require those to be locally compact:

\begin{theorem}\label{th_univ}
For a local compactum $X$, $|X|$ has the following property: for every 
map $f\colon A \to |X|$ and every FDR-embedding $i\colon A \to W$ with 
$W$ (and then also $A$) locally compact, there is a map 
$f'\colon W \to |X|$ such that $f'|_{A} \simeq f$. Moreover, any two 
such maps are homotopic.
\end{theorem}

\begin{corollary}\label{cor_class}
(1)For any local compactum $X$ and every locally compact metrizable space 
$Y$, there is a bijection $[Y,X]_{fSh} \cong [Y,|X|]$;

(2)For any local compactum $X$, the space $|X|$ is unique up to a 
homotopy equivalence.
\end{corollary}

\begin{remark}
The first part of the corollary states that $|X|$ is the classifying 
space for fine shape classes from locally compact spaces into $X$. The 
second part shows why we can talk about {\it the} space $|X|$ for a 
given $X$, regardless of the choices made in the construction of 
Definition~\ref{def_classifier}. This was first mentioned in 
Remark~\ref{rem_compact_univ}, and now we prove it in the general case.
\end{remark}

\begin{proof}[Proof (of Corollary~\ref{cor_class})]
(1)For a class $[\phi] \in [Y,X]_{fSh}$ represented by any 
$Y-X$-approaching map $\phi$, apply the theorem with $A = X$, $W$ the 
fine shape cylinder of $\phi$ (see Theorem~\ref{th_fc}), and $f$ the 
embedding (of $X$ in $W$); restricting to the 
copy of $Y$ in $W$ singles out exactly one class in $[Y,|X|]$. 
Conversely, an element of $[Y,|X|]$ can be concatenated with the fine 
shape inverse to the embedding (of $X$ into $|X|$) to obtain an element 
of $[Y,X]_{fSh}$. Uniqueness up to a homotopy implies that the two 
constructions are inverse to each other, thus defining the bijection.

(2)Routine proof: $|X|$ is the universal space characterized by 
Proposition~\ref{pr_class_ext}(2) and the theorem.
\end{proof}

\begin{proof}[Proof (of Theorem~\ref{th_univ})]
Since every point of $|X|$ is a pair $(q,\gamma)$, we have the first 
coordinate map $f_{1}\colon A \to (0,1]$, and, for every point $a \in A$, 
a point $\gamma_{a}$ of every space $|X_{\varepsilon}|$ with 
$\varepsilon \in (0,f_{1}(a)]$. Following the proof of 
Proposition~\ref{prop_compact_classifier}(3), we can 
extend the latter to assign to every point $w \in W$ a point $\gamma_{w}$ 
of every space $|X_{\varepsilon}|$ for all $\varepsilon$ small enough; 
this works because (in the notation of that proof) $H$ is continuous, 
and $RFR'$ is a composition of approaching maps. This assignment is 
continuous wherever is it defined, same as in that proof.

Now if we construct some continuous map $f_{1}'\colon W \to (0,1]$ such 
that $\gamma_{w} \in |X_{f_{1}'(w)}|$ for all $w \in W$, then this 
immediately gives us the map $f'$ we need; the homotopy 
$f'|_{A} \simeq f$ can be taken linear on the first coordinate of each 
pair, and constant on the second. To that end, in the product 
$W \times [0,1]$, consider the set $C$ of all pairs $(w,q)$ 
such that $\gamma_{w}((0,1]) \subseteq u^{-1}([q,1])$; we claim that 
$C$ contains some open neighborhood $V$ of $W \times \{0\}$. Then 
by~\cite[Lemma~18.15(a)]{Melikhov2022T}, there must be some continuous 
map $p\colon W \to [0,1]$ such that 
$p^{-1}(0)$ is actually empty, and the graph of $p$ is contained in $V$. 
From that we will be able to define $f_{1}' := p$, finishing the proof.

To verify the claim, let $w \in W$. By local compactness, $w$ has an 
open neighborhood $U$ in $W$ with compact closure $\bar{U}$, and 
$H(\bar{U} \times [0,1])$ is then compact in $N$ (again, using the 
notation from the proof of Proposition~\ref{prop_compact_classifier}(3)). 
Then by Proposition~\ref{prop_appr_equiv}(2), the image of 
$H(\bar{U} \times (0,1])$, a subset of $N\setminus W$, under the 
approaching map $RFR'$ is contained in a compact subset of $M$, and thus 
in $u^{-1}[\varepsilon,1]$ for some $\varepsilon > 0$. Therefore 
$U \times [0,\varepsilon) \subseteq C$; clearly $W \times \{0\}$ is 
covered by products of such form in $W \times [0,1]$, the union of which 
can be taken as $V$. This ends the construction of $f'$.

Finally, assume we have two maps $f'$ and $f''$ from $W$ to $|X|$ such 
that $f'|_{A} \simeq f \simeq f''|_{A}$. By using those homotopies, we 
can construct a map $g\colon W \times \{0,1\} \cup A \times [0,1] \to |X|$; 
by Proposition~\ref{prop_II_fdr} and the construction of the current proof, we can obtain a map 
$g'\colon W \times [0,1] \to |X|$ such that its restriction to the 
domain of $g$ is homotopic to $g$. Denoting such a homotopy by $G$, we 
have a concatenation of homotopies 
$G|_{W \times \{0\}} \circ g' \circ G|_{W \times \{1\}}$ proving that 
$f' \simeq f''$.
\end{proof}

\begin{remark}\label{rem_repr}
It can be seen from the proofs that if we write $i_{X}$ for the 
FDR-embedding of $X$ into $|X|$, the bijection 
$[Y,X]_{fSh} \cong [Y,|X|]$ works by sending a homotopy class $[f] \in [Y,|X|]$ 
to $[i_{X}]_{fSh}^{-1}[f]_{fS}$, and for a fine shape class 
$[\phi] \in [Y,X]_{fSh}$, the corresponding homotopy class in $[Y,|X|]$ 
turns into the fine shape class $[\phi][i_{X}]_{fSh}$ in $[Y,|X|]_{fSh}$. 
It is now easy to see that, for fixed $X$, the bijection is 
contravariantly functorial in $Y$ --- or simply functorial, if $Y$ is 
taken as an object of the (opposed) category $\mathrm{hLCM^{*}}$ of 
locally compact metrizable spaces and {\it reversed} homotopy classes of 
continuous maps. Moreover, the homotopy class $[i_{X}]$ corresponds to 
the identity fine shape class $[id_{X}]_{fSh} \in [X,X]_{fSh}$. This, of 
course, is simply a case of the Yoneda lemma, as we represent the functor 
$[-,X]_{fSh}\colon \mathrm{hLCM*} \ni Y \mapsto [Y,X]_{fSh} \in \mathrm{Set}$ 
into the usual set category by $[-,|X|]$; this, too, is why we have 
called our construction a representation. 

Still, it is not precisely a representation in the strictest 
category-theoretical sense, because $|X|$ is not locally compact, so 
$[Y,|X|]$ is not a morphism set in the same category as $Y$. This 
suggests that a further search should be conducted to find a good class 
of spaces such that, at least, for any space $X$ of that class, we can 
construct $|X|$ belonging to a (possibly different, although preferably 
the same) class from which we can also choose any space $Y$ and still 
obtain the same bijection.
\end{remark}

\bibliographystyle{bibalph}
\bibliography{references} 
\end{document}